\documentclass[12pt]{amsart}

\usepackage{graphicx, amssymb, amsthm, amsfonts, latexsym, epsfig, amsmath}
\usepackage{verbatim}
\usepackage{pstricks,pstricks-add,pst-math,pst-xkey}
\usepackage{setspace}
\newtheorem{thm}{Theorem}[section]
\newtheorem{lem}[thm]{Lemma}

\newtheorem{prop}[thm]{Proposition}

\newtheorem{conj}[thm]{Conjecture}
\theoremstyle{definition}
\newtheorem{defn}[thm]{Definition}

\newtheorem{rem}[thm]{Remark}

\renewcommand{\include}{\input}
\newcommand{\vct}[1]{\mbox{$\textbf{#1}$}}

\newcommand{\nat}{\mathbb{N}}
\newcommand{\integ}{\mathbb{Z}}

\newcommand{\cl}[1]{\overline{#1}}
\newcommand{\restr}[2]{#1\hspace{-0.15cm}\restriction_{#2}}
\newcommand{\itin}{\mbox{\texttt{\textbf{It}}}}
\newcommand{\add}{\mbox{\texttt{\textbf{A}}}}

\newcommand{\parity}[1]{%
    \ifthenelse{\equal{#1}{<}  }{\prec}{}%
    \ifthenelse{\equal{#1}{>}  }{\succ}{}%
    \ifthenelse{\equal{#1}{\le}}{\preceq}{}%
    \ifthenelse{\equal{#1}{\ge}}{\succeq}{}%
}
\DeclareMathOperator{\Orb}{Orb}  % Orbit.

\title[$\omega$-Limit Sets of Tent Maps]{On the $\omega$-Limit Sets of Tent Maps}
\author[A.D. Barwell]{Andrew D. Barwell}
\author[G. Davies]{Gareth Davies}
\author[C. Good]{Chris Good}
\address[A. D. Barwell]{School of Mathematics, University of Bristol, Howard House, Queens Avenue, Bristol, BS8 1SN, UK -- and -- School of Mathematics, University of Birmingham, Birmingham, B15 2TT, UK}
\email{A.Barwell@bristol.ac.uk}
\address[G. Davies]{University of Oxford Mathematical Institute, 24-29 St. Giles', Oxford, OX1 3LB, UK}
\email{gareth.davies@lmh.ox.ac.uk}
\address[C. Good]{School of Mathematics, University of Birmingham, Birmingham, B15 2TT, UK}
\email{cg@for.mat.bham.ac.uk}

\begin{document}

\begin{abstract}
For a continuous map $f$ on a compact metric space $(X,d)$, a set $D\subset X$ is \textit{internally chain transitive} if for every $x,y\in D$ and every $\delta>0$ there is a sequence of points $\langle x=x_0,x_1,\ldots,x_n=y\rangle$ such that $d(f(x_i),x_{i+1})<\delta$ for $0\leq i<n$. It is known that every $\omega$-limit set is internally chain transitive; in earlier work it was shown that for $X$ a shift of finite type, a closed set $D\subset X$ is internally chain transitive if and only if $D$ is an $\omega$-limit set for some point in $X$, and that the same is also true for the full tent map $T_2:[0,1]\rightarrow[0,1]$. In this paper, we prove that for tent maps with periodic critical point every closed, internally chain transitive set is necessarily an $\omega$-limit set. Furthermore, we show that there are at least countably many tent maps with non-recurrent critical point for which there is a closed, internally chain transitive set which is not an $\omega$-limit set.  Together, these results lead us to conjecture that for those tent maps with \textit{shadowing}, the $\omega$-limit sets are precisely those sets having internal chain transitivity.
\end{abstract}

\maketitle

\section{Introduction}\label{sec:intro}

In many recent texts, tent maps are cited as examples of simple maps with complicated and interesting dynamics. Furthermore, they have been the subject of many research articles

In a compact metric space $X$ with metric $d$, suppose $f:X\rightarrow X$ is a continuous map. The $\omega$-limit set $\omega(x,f)$ of a point $x\in X$ is given by
\[\omega(x,f):=\bigcap_{n\in\nat}\cl{\{f^k(x)\ :\ k\geq n\}},\]
where we may drop the dependence on the map $f$ if there is no ambiguity. So $\omega(x,f)$ is the set of accumulation points of the orbit $\Orb(x)=\Orb(x,f)$ of $x$. It is known that $\omega$-limit sets are non-empty, closed and invariant (by which we mean that $f(\omega(x))=\omega(x)$). They have been studied extensively, with particular focus on the $\omega$-limit sets of interval maps \cite{Agronsky, Balibrea, Barwell2, Barwell, Barwell3, Blokh, Bowen, BS, Bruin, Bruin2, goodknightraines2, goodknightraines, GoodRaines, goodrainessua}.

For $\delta>0$, a $\delta$-pseudo-orbit is a finite or infinite sequence of points $\langle x_0,x_1,\ldots\rangle$ such that $d(f(x_i),x_{i+1})<\delta$ for every $i\geq0$. A set $A\subset X$ is said to be \textit{internally chain transitive} (or to have \textit{internal chain transitivity}) if for every $x,y\in A$ and every $\delta>0$ there is a $\delta$-pseudo-orbit $\langle x=x_0,x_1,\ldots,x_n=y\rangle\subset A$. Chain transitivity has been studied as a dynamical property in its own right \cite{Dastjerdi, Bowen, Easton, Zheng}, and also in connection with biological systems, concerning in particular the properties of persistence and permanence \cite{Hirsch, Zheng2}. One of the first papers to link $\omega$-limit sets to the property of internal chain transitivity was \cite{Hirsch}, in which Hirsch et al show that every $\omega$-limit set is internally chain transitive, and furthermore every compact, internally chain transitive set is the $\omega$-limit set of some asymptotic pseudo-orbit. We have followed up this work in several articles: in \cite{Barwell} we show that for shifts of finite type, all closed, internally chain transitive sets are $\omega$-limit sets; in \cite{Barwell2} we show that for interval maps containing no homtervals (i.e. maps whose pre-critical points are dense in the interval, such as tent maps), all closed, internally chain transitive sets that do not contain the image of a critical point are $\omega$-limit sets; in \cite{Barwell3} we show that for the full tent map, \textit{all} internally chain transitive sets are $\omega$-limit sets. This leads us to the following conjecture:

\begin{conj}\label{conj:ICTomega_naive}
For a tent map $T:[0,1]\rightarrow [0,1]$, a closed set $L\subset [0,1]$ is internally chain transitive if and only if $L=\omega(x,T)$ for some $x\in [0,1]$.
\end{conj}

A map $f:X\rightarrow X$ has the \textit{pseudo-orbit tracing property}, or \textit{shadowing}, if for every $\epsilon>0$ there is a $\delta>0$ such that for every $\delta$-pseudo-orbit $\langle x_0,x_1,\ldots \rangle$ there is a point $y\in X$ such that $d(f^i(y),x_i)<\epsilon$ for every $i\geq0$; we say that the orbit of $y$ \textit{$\epsilon$-shadows} the pseudo-orbit. There is much evidence to suggest a link between maps with shadowing and maps for which internally chain transitive sets are necessarily $\omega$-limit sets: shifts of finite type are known to have shadowing, as is the tent map with slope equal to 2, and for both types of map we have that every internally chain transitive set is necessarily an $\omega$-limit set. We show in \cite{Barwell3} that for maps on general compact metric spaces with a type of shadowing called \textit{limit shadowing}, every internally chain transitive set is an $\omega$-limit set.

The first main result in this paper is the following:

\vspace{0.5cm}
\noindent \textbf{Theorem \ref{thm:countereg}.}
\emph{For infinitely many values $\lambda\in(\sqrt{2},2)$ there is a tent map $T:I\rightarrow I$ with slope $\lambda$ for which there exists a closed, internally chain transitive set $L$ which is not the $\omega$-limit set for any point in $I$.}\\

%which shows that at least countably infinitely many tent maps have an internally chain transitive subset which is not an $\omega$-limit set, providing a counter example to Conjecture \ref{conj:ICTomega_naive};
We use symbolic dynamics and the kneading theory to construct the maps in the proof of this result, and these maps do not have shadowing as the critical point is not recurrent. Thus we suggest the following revision of Conjecture \ref{conj:ICTomega_naive}:

\begin{conj}\label{conj:ICTomega_tent}
For a tent map $T:[0,1]\rightarrow [0,1]$ with shadowing, a closed set $L\subset [0,1]$ is internally chain transitive if and only if $L=\omega(x,T)$ for some $x\in [0,1]$.
\end{conj}

Much of the work on $\omega$-limit sets focuses on interval maps, particularly tent maps, as these maps have more uniform behaviour than smooth maps whilst maintaining very interesting and rich dynamics \cite{Barwell, Barwell3, Bruin, Bruin2, GoodRaines, goodrainessua}. However Conjecture \ref{conj:ICTomega_tent} has the following, more general formulation, to which we know of no counter example:

\begin{conj}\label{conj:ICTomega_general}
For a compact metric space $X$ and a continuous map $f:X\rightarrow X$ with shadowing, a closed set $L\subset X$ is internally chain transitive if and only if $L=\omega(x,f)$ for some $x\in X$.
\end{conj}

Many recent results on $\omega$-limit sets of interval maps have used symbolic dynamics and kneading theory (see \cite{Barwell2, Barwell, GoodRaines, goodrainessua} for examples). In this paper, we use results and techniques from symbolic dynamics and kneading theory, together with conventional analysis of interval maps, extending the theory in both areas where necessary.

In addressing Conjecture \ref{conj:ICTomega_tent}, the second main result in this paper is:

\vspace{0.5cm}
\noindent \textbf{Theorem \ref{thm:_Periodic_tent-map}.}
\emph{Suppose that $T:I\rightarrow I$ is a tent map with periodic critical point. For a closed set $D\subset I$ the following are equivalent:
\begin{enumerate}
	\item $D$ is internally chain transitive;
	\item $D$ is weakly incompressible;
	\item $D=\omega(y,T)$ for some $y\in I$.\\
\end{enumerate}}

%in which we show that for every tent map with a periodic critical point, a closed set $D\subset I$ is internally chain transitive if and only if $D$ is an $\omega$-limit set;
In the proof of this result we use a combination of conventional analysis and symbolic dynamics, together with several new results in these areas. Notice that since the critical point is periodic these maps have the shadowing property.

In \cite{Sarkovskii}, {\v{S}}arkovs$'$ki{\u\i} defines a property of invariant sets called \textit{weak incompressibility} (Definition \ref{WI} in this text), and proves that it is an inherent property of all $\omega$-limit sets. In \cite{Barwell3} we show that in compact metric spaces weak incompressibility is equivalent to internal chain transitivity, thus all theorems and conjectures concerning internal chain transitivity in this paper can also be formulated in terms of weak incompressibility.

In this paper we use standard terminology, as found in \cite{BlockCoppel, ColletEckmann, deMelo}.

\section{A Review of Symbolic Dynamics for Tent Maps}\label{symdyn}

We begin with a summary of symbolic dynamics and kneading theory for tent maps, developed by several authors  \cite{ColletEckmann, deMelo, Guck, MilnorThurston}, which we will rely upon throughout this paper. We use the notation $I$ to denote the compact interval $[0,1]$.

Let $\Omega=\{0,1,C\}$ and $T_{\lambda}:I\rightarrow I$ be a tent map with slope $\lambda\in(1,2)$ defined as usual:

\[T_{\lambda}(x)=\begin{cases}\lambda x &\mbox{ for }x\in[0,1/2]\\ \lambda(1-x) &\mbox{ for }x\in[1/2,1].\end{cases}\]

We define the symbolic dynamics for $T_{\lambda}$ with critical point $c=1/2$, including the address map $\add:I\rightarrow\Omega$, itinerary map $\itin:I\rightarrow\Omega^{\nat}$ and \textit{parity lexicographic ordering} $\prec$ as follows.

For $x\in I$ define
\[\add(x)=\begin{cases}\begin{array}{ll} 0&\mbox{ for }x\in[0,c)\\ C&\mbox{ if }x=c\\ 1&\mbox{ for }x\in(c,1]\end{array}\end{cases}\]
and
\[\itin(x)=\bigl(\add(x)\add(T_{\lambda}(x))\add(T_{\lambda}^2(x))\ldots\bigr).\]

Any finite sequence $\vct{r}$ of symbols from $\Omega$ will be referred to as a \textit{word}, and we will denote the length of the word $\vct{r}$ by $|\vct{r}|$. If a word $\vct{r}$, with $|\vct{r}|=k$ appears as the initial $k$ symbols in $\vct{s}$, where $\vct{s}$ is either another word or an infinite sequence in $\Omega^{\nat}$, we say that $\vct{r}$ is an \textit{initial $k$-segment} of $\vct{s}$, and if $\vct{r}$ appears in an arbitrary place in $\vct{s}$ we say that $\vct{r}$ is simply a \textit{$k$-segment} of $\vct{s}$; in either case we may drop the dependence upon $k$ if the length of $\vct{r}$ is unknown. Furthermore, we say that a word $\vct{r}$ is \textit{even} if it contains an even number of $1$'s and \textit{odd} otherwise. We use the symbol $^{\smallfrown}$ to represent the concatenation of sequences (either finite or infinite).

For two sequences $\vct{s}=(s_0s_1\ldots)$ and $\vct{t}=(t_0t_1\ldots)$ in $\Omega^{\nat}$ (or two words $\vct{s}=(s_0s_1\ldots s_{n-1})$ and $\vct{t}=(t_0t_1\ldots t_{n-1})$ in $\Omega^{n}$ for some $n\in\nat$), let $\restr{\vct{s}}{k}=s_0s_1\dots s_{k-1}$, then the \textit{discrepancy} of $\vct{s}$ and $\vct{t}$ is $k$ if $\restr{s}{k}=\restr{t}{k}$ and $s_k\neq t_k$. Assign a metric $d$ to $\Omega^{\nat}$ such that $d(\vct{s},\vct{t})=1/2^k$ where the discrepancy between $\vct{s}$ and $\vct{t}$ is $k$. For two sequences $\vct{s}$ and $\vct{t}$ with discrepancy $k$, define the \textit{parity lexicographic ordering}, $\prec$, on $\Omega^{\nat}$ (equally $\Omega^n$ for any $n\in\nat$) by declaring $0<C<1$, then $\vct{s}\prec \vct{t}$ provided either
\begin{enumerate}
     \item $\restr{\vct{s}}{k-1}=\restr{\vct{t}}{k-1}$ is even, and $s_k<t_k$, or
     \item $\restr{\vct{s}}{k-1}=\restr{\vct{t}}{k-1}$ is odd, and $s_k>t_k$.
\end{enumerate}

If we let the discrete topology on $\Omega$ be denoted ${\mathcal T}$ then the metric $d$ generates the Tychonoff product of ${\mathcal T}$ on the shift space $\Omega^{\nat}$. The following lemma is well-known.

\begin{lem}\label{lem:itincont}
For a tent map $T$ with critical point $c$, the itinerary map is continuous at $x$ if and only if $f^k(x)\neq c$ for all $k\geq0$.
\end{lem}

We can now define the \textit{upper-} and \textit{lower-limit itineraries} $\itin^+$ and $\itin^-$ respectively as
\begin{align*}
    \itin^+(x) &= \lim_{y\downarrow x}\itin(y)\\
    \mbox{and }\itin^-(x) &= \lim_{y\uparrow x}\itin(y),
\end{align*}
where the limit is taken in the space $\Omega^{\nat}$. Limit itineraries never contain the symbol $C$, so the limit itinerary of a point whose itinerary contains no instance of a $C$ will coincide with its itinerary \cite{Barwell2}.

For a tent map $T:I\rightarrow I$, the \textit{kneading sequence} $K_{T}$ is defined by
\[K_{T} = \sigma(\mbox{\itin}^+(c)) = \mbox{\itin}^-(T(c)),\]
where we may drop the subscript $T$ in $K_T$ if there is no ambiguity as to which map we are referring to. Notice that points in a neighbourhood of $c$ are mapped below $T(c)$, so the above definition is consistent. By Lemma \ref{lem:itincont}, the itinerary map is continuous at points whose itineraries contain no instance of a $C$, so when the critical point is not periodic we get that $K=\itin(T(c))$.

In some texts, the kneading sequence is actually defined as the itinerary of $T(c)$. We use the definition as stated here since it makes the admissibility conditions below easier to specify in the case of maps with periodic critical point.

The following conditions are well-known (see \cite{Barwell2, ColletEckmann, deMelo, MilnorThurston}) and tell us when a sequence $\vct{s}\in\Omega^{\nat}$ (or a word $\vct{s}\in\Omega^{k}$) is actually the itinerary (or initial $k$-segment of the itinerary) of some point $x\in I$.

For a tent map $T:I\rightarrow I$ with kneading sequence $K$, suppose that the sequence $\vct{s}$ satisfies the following condition:
\begin{align*}
	\mbox{either }\sigma^i(\vct{s}) &\prec K\mbox{ for every }i\geq0,\\
	\mbox{or for some $n\in\nat$, }\\
	\sigma^n(\vct{s})=\itin(c)\mbox{ and }\sigma^i(\vct{s}) &\prec K\mbox{ for every }0\leq i<n.
\end{align*}
Then there is an $x\in [0,T(c)]$ for which $\vct{s}=\itin(x)$ (we say that $\vct{s}$ is \textit{admissible}). Furthermore, if $\vct{s}=\itin(x)$ for some $x\in [0,T(c)]$, then
\begin{align*}
	\mbox{either }\sigma^i(\vct{s}) &\preceq K\mbox{ for every }i\geq0,\\
	\mbox{or for some $n\in\nat$, }\\
	\sigma^n(\vct{s})=\itin(c)\mbox{ and }\sigma^i(\vct{s}) &\preceq K\mbox{ for every }0\leq i<n.
\end{align*}
%\[\sigma^n(\vct{s})\preceq K\mbox{ for every }n\geq1\]
We say that in this case $\vct{s}$ \textit{does not violate admissibility}.

We can treat finite words in a similar way. Namely, a word $\vct{s}$ of any length $k$ for which
\begin{align*}
	\mbox{either }\sigma^i(\vct{s}) &\prec \restr{K}{k-i}\mbox{ for every }0\leq i<k,\\
	\mbox{or for some $n < k$, }\\
	\sigma^n(\vct{s})=\restr{\itin(c)}{k-n}\mbox{ and }\sigma^i(\vct{s}) &\prec \restr{K}{n-i}\mbox{ for every }0\leq i<n
\end{align*}
is such that $\vct{s}=\restr{\itin(x)}{k}$ for some $x\in [0,T(c)]$, and $\vct{s}$ is said to be \textit{admissible}. Finally, if $\vct{s}=\restr{\itin(x)}{k}$ for some $x\in [0,T(c)]$, then
\begin{align*}
	\mbox{either }\sigma^i(\vct{s}) &\preceq \restr{K}{k-i}\mbox{ for every }0\leq i<k,\\
	\mbox{or for some $n < k$, }\\
	\sigma^n(\vct{s})=\restr{\itin(c)}{k-n}\mbox{ and }\sigma^i(\vct{s}) &\prec \restr{K}{n-i}\mbox{ for every }0\leq i<n.
\end{align*}
%\[\sigma^n(\vct{s})\preceq \restr{K}{k-n}\mbox{ for every }1\leq n<k\]
Again, $\vct{s}$ is said to \textit{not violate admissibility}.

\begin{defn}\label{def:admiscond}
The conditions above are known as \textit{admissibility conditions} for sequences $\vct{s}$.
\end{defn}

\begin{rem}
The only points which do not follow the admissibility rule of Definition \ref{def:admiscond} are the points $x\in(T(c),1]$ (whenever this is a non-degenerate interval). Indeed, since every point in $(T(c),1]$ is mapped immediately into $[0,T(c)]$, points in $(T(c),1]$ have itinerary $(1\vct{s})$, where $\vct{s}$ is an infinite sequence which does not violate admissibility. We omit these points from the formal description of admissibility since every $\omega$-limit set must be a subset of the maximal invariant interval $[T^2(c),T(c)]$, we will not be concerned with the itineraries of points in $(T(c),1]$.
\end{rem}

%Lastly, in what follows we will use the symbol $\sqsubset$ to denote the concatenation of sequences and words.
For an interval map $f:I\rightarrow I$, a subinterval $J\subset I$ is called a \textit{homterval} if $c\notin f^n(J)$ for every $n\geq0$ and any local extremum $c$. Notice that tent maps $T$ with slope $\lambda\in(1,2]$ have no homterval, since every subinterval expands under $T$ until eventually it contains the critical point $c$.

Lemmas \ref{lem:itinorder}, \ref{lem:itininjective} and \ref{lem:itinint} are well-known \cite{Barwell2, Barwell, ColletEckmann, deMelo}. We state them here as they will be of use in what follows.

\begin{lem}\label{lem:itinorder}
For a continuous map $f:I\rightarrow I$ and for $x,y\in I$, \emph{$\itin(x)\prec\itin(y)$} implies that $x<y$.
\end{lem}

Due to the fact that tent maps with slope $\lambda\in(1,2]$ have no homterval, every point has a unique itinerary, and thus Lemma \ref{lem:itinorder} can be strengthened in the following way:

\begin{lem}\label{lem:itininjective}
For a tent map $T_{\lambda}$ with slope $\lambda\in(1,2]$ the itinerary map is injective. Thus for $x,y\in I$, $x<y$ if and only if \emph{$\itin(x)\prec \itin(y)$}.
\end{lem}

For $x\in I$ and $N\in\nat$, define $I_N(x):=\{y\in I\ :\ \restr{\itin(y)}{N}=\restr{\itin(x)}{N}\}$.

\begin{lem}\label{lem:itinint}
For a tent map $T_{\lambda}$ with slope $\lambda\in(1,2]$, let $x\in I$ and $N\in\nat$. Then $I_N(x)$ is an interval in $I$. Moreover if $f^n(x)=c$ for some $n\leq N$, then $I_N(x)=\{x\}$, otherwise $I_N(x)$ is an open interval.
\end{lem}

The following two results are technical observations about the relationship between sequences in $\Omega^{\nat}$ and itineraries of points.

\begin{lem}\label{lem:limitadmis}
Suppose that the sequence $\vct{s}\in\Omega^{\nat}$ is either $K=\itin^-(T(c))$ or a limit itinerary of some $x\in[0,T(c))$, either upper or lower. Then for any $n\in\nat$, $\restr{\vct{s}}{n}$ is the initial $n$-segment of some actual itinerary \emph{$\itin(y)$} for some $y\in [0,T(c))$.
\end{lem}
\begin{proof}
$\vct{s}$ is the limit point of itineraries $\itin(z)$, as $z\rightarrow x$ from either above or below, depending upon which limit we are considering. Thus as the points $z$ get closer to $x$, their itineraries will agree with $\vct{s}$ in ever larger initial segments, so for some $y\in [0,T(c))$, $\itin(y)$ will have $\vct{s}$ as its initial $n$-segment.
\end{proof}

\begin{lem}\label{lem:periodadmis}
Suppose that the tent map $T:I\rightarrow I$ has a critical point $c$ with period $m\geq3$ and kneading sequence $K$.
\begin{enumerate}
    \item Suppose that for some $N\in\nat$, a word $\vct{s}\in\{0,1\}^{N}$ is not the initial $N$-segment of the itinerary of any point in $[0,T(c)]$. Then there is a segment $\vct{r}$ of $\vct{s}$, with $|\vct{r}|=j\leq m$ for which $\vct{r}\succ \restr{K}{j}$. \label{lem:periodadmisA}
    \item Suppose that $\sigma^k(\vct{t})\succ K$ for some sequence $\vct{t}\in\{0,1\}^{\nat}$ and for some $k\geq0$. Then there is a segment $\vct{r}$ of $\vct{t}$, with $|\vct{r}|=j\leq m$ for which $\vct{r}\succ \restr{K}{j}$. \label{lem:periodadmisB}
\end{enumerate}
\end{lem}
\begin{proof}
(\ref{lem:periodadmisA}): Since $\vct{s}$ is not the initial $N$-segment of the itinerary of any point in $[0,T(c)]$, by the admissibility conditions in Definition \ref{def:admiscond}, there is a $k\in\nat$ for which
\[\sigma^k(\vct{s})\succeq\restr{K}{N-k}.\]
Furthermore, $\sigma^i(K)$ is a limit itinerary of $f^{i+1}(c)$ for every $i\geq0$%\footnote{I think this is obvious, but if not I'll put it in a separate lemma.}
, so since $\vct{s}$ is not the initial $N$-segment of the itinerary of any point in $[0,T(c)]$, by Lemma \ref{lem:limitadmis} it is not the initial $N$-segment of a limit itinerary of any point in $[0,T(c))$, nor of $K$, so in particular there must be some $k\geq0$ for which
\[\sigma^k(\vct{s})\succ\restr{K}{N-k}.\]
If the discrepancy between $\sigma^k(\vct{s})$ and $K$ is $q\leq m$, then setting $\vct{r}=\restr{\vct{s}}{q}$ we get the required result. So suppose that the discrepancy between $\sigma^k(\vct{s})$ and $K$ is $n=mi+j$, for $j<m$ and $i>0$. Then the initial $mi$-segment of $\sigma^k(\vct{s})$ is identical to that of $K$, and there is a word $\vct{t}$ which immediately follows this initial $mi$-segment of $\sigma^k(\vct{s})$, with $|\vct{t}|=j$, such that
\[\vct{t}\succ\restr{\sigma^{mi}(K)}{j}.\]
But $K$ is periodic with period $m$, so $\sigma^{mi}(K)=K$. Thus setting $\vct{r}=\vct{t}$ we get
\[\vct{r}\succ\restr{K}{j}.\]

(\ref{lem:periodadmisB}):% Since $\vct{s}$ is not the itinerary of any point in $I$, by the admissibility conditions in Definition \ref{def:admiscond}, there is a $k\in\nat$ for which
%\[\sigma^k(\vct{s})\succeq K.\]
%Furthermore, since $\sigma^k(\vct{s})\neq K$ for every $k\in\nat$, there must be some $k\geq0$ for which
%\[\sigma^k(\vct{s})\succ K.\]
If the discrepancy between $\sigma^k(\vct{t})$ and $K$ is $q\leq m$, then setting $\vct{r}=\restr{\vct{t}}{q}$ we get the required result. So suppose that the discrepancy between $\sigma^k(\vct{t})$ and $K$ is $n=mi+j$, for $j<m$ and $i>0$. Then the initial $mi$-segment of $\sigma^k(\vct{t})$ is identical to that of $K$, and there is a word $\vct{u}$ which immediately follows this initial $mi$-segment of $\sigma^k(\vct{t})$, with $|\vct{u}|=j$, such that
\[\vct{u}\succ\restr{\sigma^{mi}(K)}{j}.\]
But $K$ is periodic with period $m$, so $\sigma^{mi}(K)=K$. Thus setting $\vct{r}=\vct{u}$ we get
\[\vct{r}\succ\restr{K}{j}.\]
\end{proof}

\begin{thm}[Pre-Critical Admissibility]\label{precritadmit}
Suppose that $T:I\rightarrow I$ is a tent map with critical point $c$ and kneading sequence $K$. For any $n\in\nat$ and $\vct{s}\in\{0,1\}^{n}$, \emph{$\vct{s}^{\smallfrown}\itin(c)$} is the itinerary of some pre-critical point in $[0,T(c)]$ if and only if one of the following three conditions hold:
\begin{enumerate}
	\item $\vct{s}=0^n$;
	\item $\vct{s}\in\{0,1\}^n$ and \emph{$\sigma^k(\vct{s}^{\smallfrown} \itin(c))\prec K$} for every $0\leq k\leq n$;
	\item $c$ is in a period-$m$ orbit with itinerary $(C^{\smallfrown}\vct{t})^{\infty}$ and $\vct{s}=\sigma^k(\vct{t})$ for some $0\leq k<m-1$.
\end{enumerate}
\end{thm}
\begin{proof}
Sufficiency: we consider each case individually.

For case (1), if $\vct{s}=0^n$, then $(\vct{s}^{\smallfrown} \itin(c))$ is the itinerary of the pre-critical point, obtained by taking pre-images of $c$ within $(0,1/2)$ for $n$ pre-images.

For case (2), since $\sigma^k(\vct{s}^{\smallfrown} \itin(c))\prec K$ for every $0\leq k\leq n$, we know by the admissibility conditions that $\vct{s}$ is admissible, so for $\diamond\in\{0,1,C\}$, $(\vct{s}^{\smallfrown}\diamond)$ is also admissible. If $(\vct{s}^{\smallfrown}C)$ is admissible, then we are done, since the only itinerary which begins with $(\vct{s}^{\smallfrown}C)$ is $(\vct{s}^{\smallfrown} \itin(c))$. So assume that for $\diamond\in\{0,1\}$, $(\vct{s}^{\smallfrown}\diamond)$ is admissible. By Lemma \ref{lem:itinint}, the points whose itineraries begin with $(\vct{s}^{\smallfrown}\diamond)$ form an interval $(a,b)\subset I$; assume that $(a,b)$ is the maximal interval containing points with such itineraries. Thus for every $x\notin[a,b]$, $\itin(x)$ does not begin with $(\vct{s}\ \diamond)$, and we claim that one of $a$ or $b$ must map to $c$ under $T^n$ (the other maps to $c$ under $T^{n-j}$ for some $1\leq j\leq n$). Indeed, since $(a,b)$ is the maximal interval admitting itineraries beginning with $(\vct{s}^{\smallfrown}\diamond)$, $c\notin T^n(a,b)$, but $c\in T^n(a-\epsilon,b+\epsilon)$ for every $\epsilon>0$, hence $c\in T^n\{a,b\}$. We conclude that one of $\itin(a)$ or $\itin(b)$ is the sequence $(\vct{s}^{\smallfrown} \itin(c))$.

For case (3). if $c$ is in a period-$m$ orbit with itinerary $(C^{\smallfrown}\vct{t})^{\infty}$ and $\vct{s}=\sigma^k(\vct{t})$ for some $0\leq k<m-1$ then clearly there is point in the orbit of $c$ with itinerary $(\vct{s}^{\smallfrown} \itin(c)) = (\vct{s}^{\smallfrown} (C\ \vct{t})^{\infty})$.

Necessity: suppose that for a pre-critical point $p$ with $\itin(p)=(\vct{s}^{\smallfrown} \itin(c))$, $\vct{s}\in\{0,1\}^{n}$ is not of the form described in $(1)$ or $(3)$ (which are both instances of pre-critical points). We know that by Definition \ref{def:admiscond}, $\sigma^k(\itin(p))\preceq K$ for every $k\geq0$, so suppose that $\sigma^k(\itin(p))= K$ for some $0\leq k\leq n$. Then $\sigma^k(\itin(p)) = \itin(T^k(p)) = K$, and since the itinerary map is one-to-one by Lemma \ref{lem:itininjective}, $T^k(p)$ must equal the only point whose itinerary can be $K$, which by the definition of $K$ is $T(c)$. But then $T^{k-1}(p)=c=T^n(p)$, and since $k\leq n$ we have that $c$ is periodic -- a contradiction. Thus we are forced to conclude that $\sigma^k(\vct{s}^{\smallfrown} \itin(c))\prec K$ for all $0\leq k\leq n$.
\end{proof}

\section{A Tent Map Counter Example}\label{sec:countereg}

The following theorem is from \cite{Barwell2} and states necessary and sufficient conditions, in terms of itineraries, for one point to be in the $\omega$-limit set of another, for a class of interval map. Note that a \textit{pre-periodic} point can be either a periodic point or one which maps onto a periodic point. In other words, we allow the pre-periodic segment of a pre-periodic point's orbit to be empty.

\begin{thm}\label{thm:symomega}
Suppose that $f:I\rightarrow I$ is an interval map with no homterval. For $x,y\in I$, either
\begin{enumerate}
	\item $x$ is pre-periodic, in which case $y\in\omega(x)$ if and only if arbitrarily long initial segments of \emph{$\itin(y)$} occur infinitely often in \emph{$\itin(x)$}, or
	\item $x$ is not pre-periodic, in which case $y\in\omega(x)$ if and only if arbitrarily long initial segments of either \emph{$\itin^+(y)$} or \emph{$\itin^-(y)$} occur infinitely often in \emph{$\itin(x)$}.
\end{enumerate}
\end{thm}

This result clearly holds for tent maps whose slope $\lambda$ is in the interval $(1,2]$, as these maps have no homterval.

The main theorem in this section is in fact a family of examples which show that internal chain transitivity does not fully characterize $\omega$-limit sets in tent maps, thus providing a counter example to Conjecture \ref{conj:ICTomega_naive}. In the proof of Theorem \ref{thm:countereg} we refer to a result in \cite{ColletEckmann}, which states when a sequence of symbols is the kneading sequence of a tent map, rather than state it here explicitly, as to do so would require introducing excessive terminology.

\begin{thm}\label{thm:countereg}
For infinitely many values $\lambda\in(\sqrt{2},2)$ there is a tent map $T:I\rightarrow I$ with slope $\lambda$ for which there exists a closed, internally chain transitive set $L$ which is not the $\omega$-limit set for any point in $I$.
\end{thm}
\begin{proof}
Fix $k\in\nat,\ k\geq2$ and let $A$ be the word $10^k$ and let $B$ be the word $110$. Consider the sequence $K=A B^{\infty}$. This is the itinerary of $T(c)$ for a tent map $T$ with slope $\lambda\in(\sqrt{2},2)$ by \cite[Lemma III.1.6]{ColletEckmann}, so since it contains no symbol $C$ it is also the kneading sequence of that tent map (by the definitions of limit itinerary and kneading sequence).

Let $L$ be the set of points whose itineraries are the set $\Lambda$, where
\[
\Lambda:=\bigl\{\sigma^n(B^j C A B^{\infty})\ :\ j,n\in\nat\bigr\}.
\]
These sequences are all itineraries of points in $I$ by Definition \ref{def:admiscond} and Theorem \ref{precritadmit}. It is easy to prove that $L$ is closed and internally chain transitive. Indeed, sequences of the form $B^j C A B^{\infty}$ tend to the period-$3$ cycle as $j\rightarrow\infty$, as do sequences of the form $\sigma^n(A B^{\infty})$ as $n\rightarrow\infty$.

Suppose that $L=\omega(x)$ for some $x\in I$. Thus $T(L)=L$, so we must have that $L\subset[T^2(c),T(c)]$. We can conclude that $L=\omega(z)$, where $z\in [0,T(c)]$ and $z=T(x)$.

By Theorem \ref{thm:symomega}, we know that for $y\in I$, $y\in\omega(z)$ if and only if arbitrarily long initial segments of either the upper or lower limit itinerary of $y$ occurs infinitely often in the itinerary of $z$. Thus the itinerary of $z$ must contain infinitely many occurrences of words of the form $B^n \diamond A B^m$ for every $m,n\in\nat$, where $\diamond$ indicates a place where we are free to choose either a $0$ or a $1$ (it will be one of these infinitely often, and which one will determine whether we approximate the upper or lower limit itinerary). This forces the itinerary of $z$ to take the form
\[
\kappa=D_1 B^{n_1} \diamond A B^{m_1} D_2 B^{n_2} \diamond A B^{m_2} \ldots ,
\]
where $\{n_i\}_{i\in\nat}$ and $\{m_i\}_{i\in\nat}$ are strictly increasing sequences of positive integers, $\{D_j\}_{j\in\nat}$ are finite words, and without loss of generality $D_j$ does not begin with $B$ for any $j\in\nat$. For the sequence $\kappa$ to be the itinerary of a point $z$ as required, we need it to satisfy the admissibility conditions of Definition \ref{def:admiscond}, in particular that
\[
\sigma^j(\kappa)\preceq K\ \forall\ j\geq1.
\]
With these conditions, we consider what the first three symbols in each of the $D_j$'s can be. If the $D_j$'s begin with a particular sequence $H$ of three symbols infinitely often, then $\omega(z)$ will contain a point whose itinerary contains the sequence $B H$. We analyze each possible sequence $H$ in turn:
\begin{enumerate}
\item $D_j=0\ldots$. This case is impossible, regardless of what follows the $0$, considering the parity lexicographic order. Indeed, consider the two sequences $K$ and $\kappa$, where $D_j=0\ldots$:
\begin{align*}
K = &10^k110\ldots110110\ldots\\
\kappa =\ldots &10^k110\ldots110\underline{0}
\end{align*}
Since the number of $1$'s prior to the $\underline{0}$ is odd, we see that the iterate of $\kappa$ which starts as above does not follow the above admissibility condition.
\item $D_j=10\diamond\ldots$. This cannot occur infinitely often, else we would have a point in $\omega(z)$ whose itinerary contains the sequence $B 10\diamond$, which is none of the points in $L$.
\item $D_j=110\ldots$. This case we have excluded, since $110=B$.
\item $D_j=111\ldots$. This cannot occur infinitely often, else we would have a point in $\omega(z)$ whose itinerary contains the sequence $111$, which is none of the points in $L$.
\end{enumerate}
Thus no such (non-empty) sequence $H$ is possible infinitely often, which only leaves the possibility that $D_j=\emptyset$ a cofinite number of times. But we can also eliminate this possibility, since we have eliminated all the possible (non-empty) combinations of the three symbols (other than $B$) which can follow $B$ in $\kappa$ infinitely often. Thus no such sequence $\kappa$ exists, and we conclude that $L$ cannot be the $\omega$-limit set of any point in $I$.
\end{proof}

\begin{rem}
The initial choice of the number of $0$'s following the first $1$ in the kneading sequence in Theorem \ref{thm:countereg} gives us countably many different tent maps for which this result holds. Notice also that the critical point of these maps is not recurrent, so by \cite[Theorem 4.2]{Coven} (stated in this paper as Theorem \ref{thm:tentshad}) these maps do not have the shadowing property (unlike the full tent map, which has shadowing and for which every closed, internally chain transitive set is an $\omega$-limit set \cite{Barwell3}). So whilst Theorem \ref{thm:countereg} provides a counter example to Conjecture \ref{conj:ICTomega_naive}, it enhances the case for Conjectures \ref{conj:ICTomega_tent} and \ref{conj:ICTomega_general}.
\end{rem}

%--------------------------------------------------------------------------------------------

\section{Tent Maps with Periodic Critical Point}

We begin this section with some observations about tent maps with periodic critical point $c$. First, recall that for tent maps with slope $\lambda<1$, every point in the interval converges to $0$, and for $\lambda>2$ the critical point is mapped out of the interval $I=[0,1]$. For $\lambda=1$ every point in $[0,c]$ is fixed, and for $\lambda=2$ the critical point is eventually mapped onto $0$ which is fixed. So when considering tent maps with periodic critical point we are naturally restricting our attention to slope values $\lambda\in(1,2)$. Hence from now on we will assume this restriction on the value of $\lambda$ implicitly when referring to tent maps with periodic critical point.

For a tent map $T_{\lambda}$ with slope $\lambda$ and critical point $c$ with period $m\geq3$, let
\[
    \delta_T := \lambda^{-m}\min{\{|x - y|\ :\ x \neq y,\ x,y\in \Orb(c)\}},
\]
and let
\[
    P := \{p \in [0, 1]\ :\ T^i(p) = c \mbox{ for some $i \in \nat$}\}.
\]

For each $p \in P$ let $n_p \in \nat$ be least such that $T^{n_p}(p) = c$, and for $n\in\nat$ let
\[
    P_{n} := \{p \in P\ :\ n_p < \max\{n,2m\}\},
\]
so that $Orb(c)\subset P_n$ for every $n\in\nat$.

\begin{lem}\label{lem:evenkneadingseq}
Suppose that the tent map $T$ has periodic critical point $c$ of period $m\geq3$ and kneading sequence $K=(D^{\smallfrown}\diamond)^{\infty}$, for $\diamond\in\{0,1\}$ and $D\in\{0,1\}^{m-1}$. Then $D^{\smallfrown}\diamond$ is even; in other words, $\diamond=1$ if $D$ is odd and $\diamond=0$ if $D$ is even.
\end{lem}
\begin{proof}
Suppose that $D^{\smallfrown}\diamond$ is odd. Let $x\in \bigl(T(c)-\delta_T, T(c)\bigr)$ and let $d_x:=|x-T(c)|$. By the definition of $\delta_T$, $c\notin T^i\bigl[x,T(c)\bigr)$ for every $i\leq m$. Then since $D^{\smallfrown}\diamond$ is odd and the slope of the map is $\lambda>1$,
\begin{align*}
T^m(x) &=x+\lambda^m d_x\\
 &>T(c)
\end{align*}
which is not possible. Thus $D^{\smallfrown}\diamond$ must be even.
\end{proof}

As a consequence of Lemma \ref{lem:evenkneadingseq}, for tent maps $T$ with periodic critical point $c$ of period $m\geq3$, $T^m[c-\delta_T,c+\delta_T]\subset(0,c]$ if $\restr{K}{m-1}$ is even, and $T^m[c-\delta_T,c+\delta_T]\subset[c,1)$ if $\restr{K}{m-1}$ is odd. We exploit this property in the following definition.

\begin{defn}
For a tent map $T$ with periodic critical point $c$ of period $m\geq3$, let $S=\bigl\{[c - \delta_T, c), (c, c + \delta_T]\bigr\}$ and define the \textit{accessible side of $c$}, $A\in S$, to be
\[
A:=
\begin{cases}
    [c - \delta_T, c) &\mbox{ if }T^m[c - \delta_T, c + \delta_T] \subset [c,1)\\
    (c, c + \delta_T] &\mbox{ if }T^m[c - \delta_T, c + \delta_T] \subset (0,c]
\end{cases}
\]
Then the \textit{hidden side of $c$}, $H$, is defined to be the element of $S$ which is not $A$.
\end{defn}

In this section we prove that a closed, internally chain transitive set of a tent-map whose critical point is periodic is necessarily an $\omega$-limit set.  %I've been deliberately vague about the basic mechanics of the symbolics because I'm fairly sure they work and I would not like to write lemmas establishing basic facts when they probably exist already in other papers.  Please forgive the messy notation.
This result follows on from a series of results including the following from \cite{Barwell2}, which relates $\omega$-limit sets of maps with no homterval to internally chain transitive sets which do not contain the image of any critical point:

\begin{thm}\label{thm:tentnocrit}
Suppose that $f:I\rightarrow I$ is an interval map having critical points $c_1,\ldots,c_k$ and no homterval, and $D\subset I$ is closed and does not contain $f(c_i)$ for any $1\leq i\leq k$. Then $D$ is internally chain transitive if and only if $D=\omega(x,f)$ for some $x\in I$.
\end{thm}

Since tent maps with slope $\lambda\in(1,2)$ have no homterval, such maps satisfy the hypothesis of Theorem \ref{thm:tentnocrit}.

We add to the above description of symbolic dynamics for tent maps the definition of a \textit{signature sequence} \cite{MilnorThurston}, which indicates whether the slope of the map is positive or negative at each iterate of the critical point $c$.

For a tent map $T$ with critical point $c$ and kneading sequence $K=K_0 K_1 K_2 \ldots$, we define the signature sequence $\rho=\rho_0 \rho_1 \rho_2 \ldots$ as follows:
\[\rho_0 = -1\]
\[\begin{array}{lll}
 \rho_{n+1} & = \rho_n &\mbox{ if }K_n = 0\\
 \rho_{n+1} & = -\rho_n &\mbox{ if }K_n = 1\\
 \rho_{n+1} & = -1 &\mbox{ if }K_n = C
\end{array}\]

The following theorem, proved in \cite{Coven}, has been referred to already in this paper, and we state it now explicitly as it will be required in what follows:

\begin{thm}\label{thm:tentshad}
Suppose that $T:I\rightarrow I$ is a tent map with slope $\lambda\in(1,2)$ and kneading sequence $K=(K_0 K_1 K_2 \ldots)\in\{0,1,C\}^{\nat}$. Then $T$ has the shadowing property if and only if for every $\epsilon>0$ there is an $n\in\nat$ such that $|T^n(c)-c|<\epsilon$ and either $T^n(c)=c$, or $\rho_n = +1$ if $K_n = 0$ and $\rho_n = -1$ if $K_n = 1$.
\end{thm}

The following property was introduced by {\v{S}}arkovs$'$ki{\u\i}, who proved that it is an inherent property of $\omega$-limit sets (this fact was also observed in \cite{BlockCoppel}). The term \textit{weak incompressibility} appears in \cite{Balibrea} and we adopt this term here.

\begin{defn}\label{WI}
For a continuous map $f$ on a compact metric space $X$, a closed set $L\subset X$ is said to have \textit{weak incompressibility} (or is said to be \textit{weakly incompressible}) if for any non-empty subset $U\subsetneq L$ which is open in $L$, $\cl{f(U)}\nsubseteq U$.
\end{defn}

Theorem \ref{thm:omega_WI} is due to {\v{S}}arkovs$'$ki{\u\i}, and is from \cite{Sarkovskii}.

\begin{thm}\label{thm:omega_WI}
Suppose that $f:X\rightarrow X$ is a continuous map on the compact metric space $X$. Then for every $x\in X$, $\omega(x,f)$ has weak incompressibility.
\end{thm}

The following two results are from \cite{Barwell3} and will also be used in the proof of the main result in this section.

\begin{prop}\label{prop:ICTinv}
Let $(X,d)$ be a compact metric space, and $f:X\rightarrow X$ be continuous. If $L$ is a closed, internally chain transitive subset of $X$, then $L$ is invariant; in other words $f(L)=L$.
\end{prop}

\begin{thm}\label{thm:WI=ICT}
Let $(X,d)$ be a compact metric space, $f:X\rightarrow X$ be continuous and let $L$ be a closed, nonempty subset of $X$. Then $L$ is internally chain transitive if and only if $L$ is weakly incompressible.
\end{thm}

If $D\subset I$ is closed and internally chain transitive, then by Proposition \ref{prop:ICTinv} we know that $T(D)=D$, but this does not imply that $T^{-1}(D)=D$. The following proposition shows that we can expand such a set $D$ in a certain way to obtain another closed and internally chain transitive set which contains more elements of $T^{-1}(D)$. As usual, for $x\in I$ and $\epsilon>0$ we define $B_{\epsilon}(x):=\{y\in I\ :\ |y-x|<\epsilon\}$.

\begin{prop}\label{prop:ICTextn}
Suppose that $T:I\rightarrow I$ is a tent map with slope $\lambda$ having critical point $c$ with period $m\geq3$ and accessible side $A$. Suppose also that $D\subset I$ is a closed and internally chain transitive set which contains $c$ and for which $c$ is not isolated in $D$.

For each $n \in \nat$ and each $p\in P$, there is a set $D_{n, p} \subset \cl{B_{2^{-n}\lambda^{-n_p}\delta_T}(p)}$ defined as
\[
    D_{n, p} := T^{-n_p}\Bigl(\cl{B_{2^{-n}\delta_T}(c)} \cap A \cap D\Bigr) \cap \cl{B_{2^{-n}\lambda^{-n_p}\delta_T}(p)}.
\]
Moreover, for each $n\in\nat$ and $p\in P$ the following are true:
\begin{enumerate}
    \item $T^{n_p}(D_{n,p}) = \cl{B_{2^{-n}\delta_T}(c)} \cap A \cap D$;
    \item $D_{n,p}\cup\{p\} = T^{-n_p}\Bigl(\cl{B_{2^{-n}\delta_T}(c)} \cap \cl{A} \cap D\Bigr) \cap \cl{B_{2^{-n}\lambda^{-n_p}\delta_T}(p)}$, so this set is compact;
    \item $D_n := D \cup \Bigl(\bigcup_{p \in P \cap D}{D_{n, p}}\Bigr)$ is closed and internally chain transitive.
\end{enumerate}
\end{prop}
\begin{proof}
Since $c$ is not isolated in $D$, we have that every point in $Orb(c)$ is not isolated in $D$, including $T^m(c)=c$. In particular, $c$ is not isolated in $D\cap A$, so the set $D_{n,p}$ is well defined.

Property (1) follows from the definition of $D_{n,p}$. To see property (2) notice that $\cl{A}\setminus A = \{c\}$, so by the definition of $D_{n,p}$,
\[
T^{-n_p}\Bigl(\cl{B_{2^{-n}\delta_T}(c)} \cap \cl{A} \cap D\Bigr) \cap \cl{B_{2^{-n}\lambda^{-n_p}\delta_T}(p)}\setminus D_{n,p} = \{p\}.
\]
Thus $D_{n,p} \cup \{p\}$ is as defined, and compactness follows immediately.

To prove property (3), fix any $n\in\nat$ and let $\langle a_i\ |\ i \in \nat\rangle$ be any sequence in $D_n$. Either there exists $K \in \nat$ such that
\[
    \{a_i\ :\ i \in \nat\} \subset D \cup \bigcup_{p \in P_K \cap D}{D_{n, p}}
\]
or there exists a subsequence $\langle b_i\ |\ i \in \nat\rangle$ of $\langle a_i\ |\ i \in \nat\rangle$ and a sequence $\langle p_i\ |\ i \in \nat\rangle$ in $P \cap D$ such that for all $i \in \nat$:
\begin{itemize}
\item $n_{p_{i + 1}} > n_{p_i}$;
\item $\{b_i, p_i\} \subset D_{n, p_i}$.
\end{itemize}
In the former case, $D \cup \left(\bigcup_{p \in P_K \cap D}{D_{n, p}}\right)$ is compact (since $D_{n,p}\cup\{p\}$ is compact for each $p\in D$) and so $\langle a_i\ |\ i \in \nat\rangle$ has a convergent subsequence.  In the latter case, for all $i \in \nat$ $|b_i - p_i| \leq 2\lambda^{-n_{p_i}}\delta$ which converges to $0$ as $i \rightarrow \infty$.  $D$ is closed in $[0, 1]$ and hence is sequentially compact so $\langle p_i\ |\ i \in \nat\rangle$ has a subsequence which converges to a point $x \in D \subset D_n$.  It follows that $\langle a_i\ |\ i \in \nat\rangle$ has a subsequence which converges to $x$.  In either case, $\langle a_i\ |\ i \in \nat\rangle$ has a convergent subsequence, and hence $D_n$ is sequentially compact and is therefore closed.

Let $x \in D_n \setminus D$, $y \in D$, and fix $\epsilon > 0$.  To show that $D_n$ is internally chain transitive it suffices to show that there exist $\epsilon$-pseudo-orbits from $x$ to $y$ and from $y$ to $x$, both completely contained in $D_n$.  Let $p \in P \cap D$ be such that $x \in D_{n, p}$.  Then $T^{n_p}(x) \in D$ and $\langle T^i(x)\ |\ i \leq n_p\rangle$ is an $\epsilon$-pseudo-orbit from $x$ to $T^{n_p}(x)$.  Because $D$ is internally chain transitive we can find an $\epsilon$-pseudo-orbit from $T^{n_p}(x)$ to $y$.  Concatenating the two, we obtain an $\epsilon$-pseudo-orbit from $x$ to $y$ completely contained in $D_n$.  To go from $y$ to $x$ note that by invariance of $D$, for each $p' \in P \cap D$ we can find $p'' \in P \cap D$ such that $T(p'') = p'$.  Thus we can find a sequence $\langle p_i\ |\ i \in \nat\rangle$ in $P \cap D$ such that $p_0 = p$ and for each $i \in \nat$, $T(p_{i + 1}) = p_i$.  Let $j \in \nat$ be such that $\epsilon > \lambda^{-j}$.  It follows from the argument above and the definitions of $D_{n, p}$ and $D_{n, p_j}$ that there exists $z \in D_{n, p_j}$ such that $T^j(z) = x$.  Because $z \in D_{n, p_j}$ and $n_{p_j}>j$, $|z - p_j| < \epsilon$ so there exists an $\epsilon$-pseudo-orbit from $y$ to $z$ and hence there exists an $\epsilon$-pseudo-orbit from $y$ to $x$ completely contained in $D_n$, as required.
\end{proof}

\begin{prop}\label{prop:orbit_seg}
Let $T$ be a tent map having critical point $c$ and let $\{q_n\}_{n\in\nat}\subset\nat$. For each $n\in\nat$ suppose there are points $x_n\in I$, open intervals $B_n\subset I$ and integers $J_n\in\nat$ such that
\begin{enumerate}
	\item $q_n\leq J_n$ for each $n$;
	\item $T^{J_n}(x_n),x_{n+1}\in B_n$;
	\item for every $0\leq j\leq J_n$, $c\notin T^j(B_n)$, so $T^j$ maps $B_n$ homeomorphically onto its image for every $0\leq j\leq J_n$.
\end{enumerate}
Let $\langle a_i : 0\leq i\in\integ\rangle$ be the sequence
\[
\langle x_1,T(x_1),\ldots,T^{J_1-1}(x_1),x_2,T(x_2),\ldots,T^{J_2-1}(x_2),x_3,\ldots\rangle.
\]
Then for every $n\in\nat$, $n>1$, and every $t\geq\sum_{k<n}J_{k}$ we have
\[
\restr{\itin(a_t)}{q_n} = \bigl(\add(a_i)\bigr)_{i=t}^{t+q_n-1}.
\]
\end{prop}
\begin{proof}
First notice that since $T^{J_n}(x_n),x_{n+1}\in B_n$, and $c\notin T^j(B_n)$ for every $j\leq J_n$, the itineraries of $T^{J_n}(x_n)$ and $x_{n+1}$ agree on their first $J_n\geq q_n$ places.

For each $n\in\nat$ let $t_n=\sum_{k<n}J_{k}$% and let $\gamma=\bigl(\add(a_i)\bigr)_{i\in\nat}$
. Pick $n\in\nat$ and suppose that $t\geq t_n$. If $t=t_{n'}$ for some $n' \geq n$ then $\itin(a_t) = \itin(x_{n'})$, so since $c\notin T^j(B_{n'})\ni T^j(x_{n'})$ for every $0\leq j\leq J_{n'}$,
\[
\restr{\itin(x_{n'})}{J_{n'}} = \restr{\itin(a_t)}{J_{n'}} = \bigl(\add(a_i)\bigr)_{i = t}^{J_{n'}-1},
\]
and as $J_{n'} \geq q_n$ we are done.

If $t=t_{n'}+r$ for some $n' \geq n$ and $0<r < J_{n'+1}$, then %$\sigma^{t-1}(\gamma)$
$\bigl(\add(a_i)\bigr)_{i\geq t}$ and $\itin(a_t)$ follow the orbit of $T^r(x_{n'})$ until it reaches $x_{n'+1}$, after which %$\sigma^{t-1}(\gamma)$
$\bigl(\add(a_i)\bigr)_{i\geq t}$ follows $x_{n'+1}$ and $\itin(a_t)$ follows $T^{J_{n'}}(x_{n'})$, and both of these agree for at least $J_{n'+1} \geq q_n$ places, so we use reasoning as above to deduce that $\restr{\itin(a_t)}{q_n} = \bigl(\add(a_i)\bigr)_{i=t}^{t+q_n-1}$.
\end{proof}

We can now prove our main result, which shows that for tent maps with periodic critical point (and thus having shadowing), closed internally chain transitive sets are necessarily $\omega$-limit sets.

\begin{thm}\label{thm:_Periodic_tent-map}
Suppose that $T:I\rightarrow I$ is a tent map with periodic critical point. For a closed set $D\subset I$ the following are equivalent:
\begin{enumerate}
	\item $D$ is internally chain transitive;\label{pertent1}
	\item $D$ is weakly incompressible;\label{pertent2}
	\item $D=\omega(y,T)$ for some $y\in I$.\label{pertent3}
\end{enumerate}
\end{thm}

\begin{proof}
For (\ref{pertent3}) $\Rightarrow$ (\ref{pertent2}) and (\ref{pertent2}) $\Rightarrow$ (\ref{pertent1}) we refer to Theorems \ref{thm:omega_WI} and \ref{thm:WI=ICT} respectively.

For (\ref{pertent1}) $\Rightarrow$ (\ref{pertent3}), let $D \subset I$ be closed and internally chain transitive; by Proposition \ref{prop:ICTinv}, $D$ is invariant.  Thus either $D = \{0\}$ or $D \subset [T^2(c), T(c)]$; because $\{0\}$ is an $\omega$-limit set we may assume without loss of generality that $D \subset [T^2(c), T(c)]$.

By Theorem \ref{thm:tentnocrit} if $c \not\in D$ we have that $D$ is an $\omega$-limit set, so suppose from now on that $c\in D$.

Let $c$ have period $m\geq3$. If $c$ is isolated in $D$, then by invariance of $D$, $T^j(c)$ is also a pre-image of $c$ so is isolated in $D$.  Because $D$ is internally chain transitive, it follows that $D = \Orb(c)$, which is an $\omega$-limit set. Thus, without loss of generality we may assume that $c$ is not isolated in $D$, which gives us that no element of $Orb(c)$ is isolated in $D$, thus it follows that $c$ is not isolated in $A\cap D$.

For each $n \in \nat$ and each $p\in P$ let $\delta_T$, $n_p$ and $P_n$ be as defined above, and let $D_n$ be constructed as in Proposition \ref{prop:ICTextn}. Thus $D_n$ is closed and internally chain transitive by property (3) of \ref{prop:ICTextn}, and is thus invariant by Proposition \ref{prop:ICTinv}.

Since $c$ is not isolated in $A\cap D$, we have constructed $D_n$ so that for each $p\in P\cap D$ it contains points which accumulate on $p$, meaning no element of $D_n \cap P$ is isolated in $D_n$.  $D_n \setminus B_{2^{-(n + 1)}}(P_n)$ is compact so we can find a finite $F \subset D_n \setminus P_n$ such that
\[
    D_n \setminus B_{2^{-(n + 1)}}(P_n) \subset B_{2^{-n}}(F).
\]
For each $p \in P_n$ pick $x_p \in D_{(n + 1), p}$.  Then
\[
    F_n := F \cup \{x_p\ :\ p \in P_n\}
\]
is a finite subset of $D_n \setminus P_n$ such that
\[
    D_n \subset B_{2^{-n}}(F_n).
\]
Write $F_n = \{b_{n, i}\ :\ i \leq I_n\}$.

For each $n\in\nat$ let $\epsilon_n > 0$ be such that $B_{\epsilon_n}(F_n) \cap P_n =\emptyset$ and note that $\epsilon_n < 2^{-n}$.  $T$ has shadowing by Theorem \ref{thm:tentshad}, so there exists $\eta_n < \epsilon_n$ such that every $\eta_n$-pseudo-orbit is $\epsilon_n$-shadowed.  Because $D_n$ is internally chain transitive we can find a $\eta_n$-pseudo-orbit from $b_{n, 0}$ to $b_{n, I_n}$ through each $b_{n, i}$ such that every member of the pseudo-orbit is in $D_n$, and then find $c_n \in [0, 1]$ and $J_n \in \nat$ such that $\langle T^j(c_n)\ |\ j \leq J_n\rangle$ $\epsilon_n$-shadows this $\eta_n$-pseudo-orbit.

For each $n \in \nat$, find a $\eta_{n + 1}$-pseudo-orbit from $b_{n, I_n}$ to $b_{n + 1, 0}$ such that every member of the pseudo-orbit is in $D_n$ and find $d_n \in I$ and $K_n \in \nat$ such that $\langle T^k(d_n)\ |\ k \leq K_n\rangle$ $\epsilon_{n + 1}$-shadows this pseudo-orbit.

Since $F_n$ is a subset of the invariant set $D_n$ and $B_{\epsilon_n}(F_n) \cap P_n =\emptyset$ and $T(c), T^2(c) \in P_n$ we know that $c_n, d_n \in [T^2(c), T(c)]$ for every $n \in \nat$.

Let
\[
    \langle a_t\ |\ t \geq0\rangle = \langle c_1,T(c_1),\ldots,T^{J_1-1}(c_1),d_1,T(d_1),\ldots,T^{K_1-1}(d_1),c_2,\ldots\rangle.
\]
For each $t \in \nat$, let
\[
    \alpha^t := \langle \alpha^t_i\ |\ i \geq 0\rangle = \itin(a_t)
\]
and let
\[
    \gamma := \langle \gamma_i\ |\ i \in \nat\rangle = \langle \alpha^i_0\ |\ i \in \nat\rangle.
\]
Each $a_t$ is some (possibly trivial) forward image of a $c_n$ or a $d_n$ and consequently must lie in $[T^2(c), T(c)]$.  It follows by Definition \ref{def:admiscond} that for each $\alpha^t$, $\sigma^k(\alpha^t) \preceq K$ for every $k\geq0$.

Fix $n \in \nat$.  We have that $B_{\epsilon_n}\left(b_{n, I_n}\right) \cap P_n =\emptyset$ and hence $T^{J_n}(c_n) \notin P_n$.  In particular, $T^j(x) \notin P_n$ for any $x\in B_{\epsilon_n}\left(b_{n, I_n}\right)$ and any $j \le J_n$. Because $\Orb(c) \subset P_n$ for each $n \in \nat$, it follows that $\gamma \in \{0, 1\}^{\nat}$.

For each $n \in \nat$ let $t_n = \sum_{n' < n}{J_{n'} + K_{n'}}$, so that $a_{t_n} = c_n$ for each $n \in \nat$. Thus $\alpha^{t_n}=\itin(a_{t_n})=\itin(c_n)$. Also let $q_n=\max\{n,2m\}$. We claim that for each $t\geq t_n$, $\restr{\alpha^t}{q_n} = \restr{\sigma^t(\gamma)}{q_n}$. To see this, first notice that both $d_n$ and $T^{J_n}(c_n)$ are in $B_{\epsilon_n}(b_{n,I_n})$, and $B_{\epsilon_n}(b_{n,I_n})\cap P_n=\emptyset$. So since no point in $B_{\epsilon_n}(b_{n,I_n})$ maps to $c$ in any of their first $J_n$ iterations, Proposition \ref{prop:orbit_seg} tells us that $\restr{\alpha^t}{q_n} = \restr{\sigma^t(\gamma)}{q_n}$ since $q_n<J_n$, which proves the claim.

To prove admissibility of $\gamma$, suppose firstly that $\sigma^l(\gamma) = K$ for some $l \in \nat$ and fix $x \in D$.  For each $n \in \nat$ there exists $i \leq I_n$ such that $|b_{n, i} - x| < 2^{-n}$.  Then there exists $j \leq J_n$ such that $|T^j(c_n) - b_{n, i}| < \epsilon_n$ and, of course, there exists $r_n \geq t_n$ such that $a_{r_n} = T^j(c_n)$.  We have that
\[
    |a_{r_n} - x| < 2^{-n} + \epsilon_n
\]
and that by our claim above, for every $n \in \nat$
\[
    \restr{\itin(a_{r_n})}{n} = \restr{\alpha^{r_n}}{n} = \restr{\sigma^{r_n}(\gamma)}{n}.
\]
For some $n'\in\nat$, $t_n'\geq l$, so for each $n\geq n'$
\[\restr{\itin(a_{r_n})}{n} = \restr{\sigma^{r_n}(\gamma)}{n} = \restr{\sigma^{r_n-l}}{K}\]
It follows that either the upper- or lower- limit itinerary of $x$ is an iterate of $K$, which forces $x \in \Orb(c)$.  %\textbf{I probably should add more detail here.}
But then since $x\in D$ was arbitrary, $D \subset \Orb(c)$, a contradiction.

If $\sigma^l(\gamma) \succ K$ for some $l \in \nat$ then by Lemma \ref{lem:periodadmis} there exists $s \geq0$ such that $\sigma^{l + s}(\gamma) \succ K$ and $\restr{\sigma^{l + s}(\gamma)}{2m} \neq \restr{K}{2m}$.  We also have that $\restr{\alpha^{l + s}}{2m} = \restr{\sigma^{l + s}(\gamma)}{2m}$.  Thus, $\alpha^{l + s} \succ K$.  By admissibility conditions \ref{def:admiscond}, we have that $\alpha^{l + s} \notin \itin(I)$, a contradiction.

Thus $\sigma^l(\gamma) \prec K$ for each $l \in \nat$.  by admissibility conditions \ref{def:admiscond} there must exist $y \in [0, 1]$ such that $\itin(y) = \gamma$.

It remains to show that $D = \omega(y)$.

To see that $D \subset \omega(y)$ fix $x \in D$ and $n \in \nat$.  As before, there exists $r_n \ge t_n$ such that $|a_{r_n} - x| < 2^{-n} + \epsilon_n$ and that
\[
    \restr{\itin(a_{r_n})}{n} = \restr{\alpha^{r_n}}{n} = \restr{\sigma^{r_n}(\gamma)}{n} = \restr{\itin(T^{r_n}(y))}{n}.
\]
Thus, because two points of the interval whose itineraries agree on the first $n$ places cannot be more than a distance of $\lambda^{-n}$ apart,% \textbf{lemma?},
\[
    |T^{r_n}(y) - x| < 2^{-n} + \epsilon_n + \lambda^{-n},
\]
which converges to $0$ as $n \rightarrow \infty$.  Thus, $x$ is an accumulation point of $\Orb(y)$; i.e. $x \in \omega(y)$.

To see that $\omega(y) \subset D$ fix $n \in \nat$ and $r_n \ge t_n$.  By construction, $a_{r_n}$ (which is within $\lambda^{-n}$ of $T^{r_n}(y)$) is within $\epsilon_n$ of $D_n$.  Each point of $D_n$ is within $2^{-n}$ of $D$.  Thus, there exists $x_n \in D$ such that
\[
    |T^{r_n}(y) - x_n| < \lambda^{-n} + \epsilon_n + 2^{-n},
\]
which converges to $0$ as $n \rightarrow \infty$.  It follows that every accumulation point of $\Orb(y)$ has distance $0$ from $D$ and so $\omega(y) \subset \cl{D} = D$.
\end{proof}

\bibliographystyle{plain}
\bibliography{BibtexW-limitsets}

\end{document}